\documentclass{amsart}

\usepackage{amsthm}
\usepackage[leqno]{amsmath}
\usepackage[all]{xy} \SelectTips{eu}{}
\usepackage{latexsym,amsfonts,amssymb}
\usepackage{stmaryrd}
\usepackage{hyperref}
\usepackage{mathptmx}

\newcommand{\numberseries}{\bfseries}   

\newlength{\thmtopspace}                
\newlength{\thmbotspace}                
\newlength{\thmheadspace}               
\newlength{\thmindent}                  

\newtheoremstyle{fixed bf head,slanted body}
                {\thmtopspace}{\thmbotspace}{\slshape}
                {\thmindent}{\bfseries}{.}{\thmheadspace}
                {{\numberseries \thmnumber{#2\;}}\thmname{#1}\thmnote{ (#3)}}

\newtheoremstyle{fixed bf head,upright body}
                {\thmtopspace}{\thmbotspace}{\upshape}
                {\thmindent}{\bfseries}{.}{\thmheadspace}
                {{\numberseries \thmnumber{#2\;}}\thmname{#1}\thmnote{ (#3)}}

\newtheoremstyle{numbered paragraph}
                {\thmtopspace}{\thmbotspace}{\upshape}
                {\thmindent}{\upshape}{}{\thmheadspace}
                {{\numberseries \thmnumber{#2.}}}

\theoremstyle{fixed bf head,slanted body}
\newtheorem{res}{}[section]
\newtheorem{thm}[res]{Theorem}          \newtheorem*{thm*}{Theorem}
\newtheorem{prp}[res]{Proposition}      \newtheorem*{prp*}{Proposition}
\newtheorem{cor}[res]{Corollary}        \newtheorem*{cor*}{Corollary}
\newtheorem{lem}[res]{Lemma}            \newtheorem*{lem*}{Lemma}

\theoremstyle{fixed bf head,upright body}
\newtheorem{stp}[res]{Setup}            \newtheorem*{stp*}{Setup}
\newtheorem{dfn}[res]{Definition}       \newtheorem*{dfn*}{Definition}
     \newtheorem*{con*}{Construction}
\newtheorem{obs}[res]{Observation}      \newtheorem*{obs*}{Observation}
           \newtheorem*{rmk*}{Remark}
\newtheorem{exa}[res]{Example}          \newtheorem*{exa*}{Example}
         \newtheorem*{qst*}{Question}

\theoremstyle{numbered paragraph}
\newtheorem{ipg}[res]{}

\newlength{\thmlistleft}        
\newlength{\thmlistright}       
\newlength{\thmlistpartopsep}   
\newlength{\thmlisttopsep}      
\newlength{\thmlistparsep}      
\newlength{\thmlistitemsep}     

\newcounter{eqc} 
  {\end{list}}%

\newcounter{prt}
  {\end{list}}%

\newcommand{\pgref}[1]{\ref{#1}}
\newcommand{\thmref}[2][Theorem~]{#1\pgref{thm:#2}}
\newcommand{\corref}[2][Corollary~]{#1\pgref{cor:#2}}
\newcommand{\prpref}[2][Proposition~]{#1\pgref{prp:#2}}
\newcommand{\lemref}[2][Lemma~]{#1\pgref{lem:#2}}
\newcommand{\obsref}[2][Observation~]{#1\pgref{obs:#2}}

\newcommand{\dfnref}[2][Definition~]{#1\pgref{dfn:#2}}
\newcommand{\exaref}[2][Example~]{#1\pgref{exa:#2}}

\newcommand{\secref}[2][Section~]{#1\ref{sec:#2}}

\renewcommand{\eqref}[1]{(\pgref{eq:#1})}

\makeatletter
\def\@nobreak@#1{\mathchoice%
  {\nobreakdef@\displaystyle\f@size{#1}}%
  {\nobreakdef@\nobreakstyle\tf@size{\firstchoice@false #1}}%
  {\nobreakdef@\nobreakstyle\sf@size{\firstchoice@false #1}}%
  {\nobreakdef@\nobreakstyle\ssf@size{\firstchoice@false #1}}%
  \check@mathfonts}%
\def\nobreakdef@#1#2#3{\hbox{{%
                    \everymath{#1}%
                    \let\f@size#2\selectfont%
                    #3}}}%
\makeatother

\DeclareSymbolFont{usualmathcal}{OMS}{cmsy}{m}{n}
\DeclareSymbolFontAlphabet{\mathcal}{usualmathcal}

\DeclareSymbolFont{letters}{OML}{txmi}{m}{it}

\DeclareMathSymbol{\alpha}{\mathord}{letters}{"0B}
\DeclareMathSymbol{\beta}{\mathord}{letters}{"0C}
\DeclareMathSymbol{\gamma}{\mathord}{letters}{"0D}
\DeclareMathSymbol{\delta}{\mathord}{letters}{"0E}
\DeclareMathSymbol{\epsilon}{\mathord}{letters}{"0F}
\DeclareMathSymbol{\zeta}{\mathord}{letters}{"10}
\DeclareMathSymbol{\eta}{\mathord}{letters}{"11}
\DeclareMathSymbol{\theta}{\mathord}{letters}{"12}
\DeclareMathSymbol{\iota}{\mathord}{letters}{"13}
\DeclareMathSymbol{\kappa}{\mathord}{letters}{"14}
\DeclareMathSymbol{\lambda}{\mathord}{letters}{"15}
\DeclareMathSymbol{\mu}{\mathord}{letters}{"16}
\DeclareMathSymbol{\nu}{\mathord}{letters}{"17}
\DeclareMathSymbol{\xi}{\mathord}{letters}{"18}
\DeclareMathSymbol{\pi}{\mathord}{letters}{"19}
\DeclareMathSymbol{\rho}{\mathord}{letters}{"1A}
\DeclareMathSymbol{\sigma}{\mathord}{letters}{"1B}
\DeclareMathSymbol{\tau}{\mathord}{letters}{"1C}
\DeclareMathSymbol{\upsilon}{\mathord}{letters}{"1D}
\DeclareMathSymbol{\phi}{\mathord}{letters}{"1E}
\DeclareMathSymbol{\chi}{\mathord}{letters}{"1F}
\DeclareMathSymbol{\psi}{\mathord}{letters}{"20}
\DeclareMathSymbol{\omega}{\mathord}{letters}{"21}
\DeclareMathSymbol{\varepsilon}{\mathord}{letters}{"22}
\DeclareMathSymbol{\vartheta}{\mathord}{letters}{"23}
\DeclareMathSymbol{\varpi}{\mathord}{letters}{"24}
\DeclareMathSymbol{\varrho}{\mathord}{letters}{"25}
\DeclareMathSymbol{\varsigma}{\mathord}{letters}{"26}
\DeclareMathSymbol{\varphi}{\mathord}{letters}{"27}

\DeclareMathSymbol{\Gamma}{\mathord}{letters}{"00}
\DeclareMathSymbol{\Delta}{\mathord}{letters}{"01}
\DeclareMathSymbol{\Theta}{\mathord}{letters}{"02}
\DeclareMathSymbol{\Lambda}{\mathord}{letters}{"03}
\DeclareMathSymbol{\Xi}{\mathord}{letters}{"04}
\DeclareMathSymbol{\Pi}{\mathord}{letters}{"05}
\DeclareMathSymbol{\Sigma}{\mathord}{letters}{"06}
\DeclareMathSymbol{\Upsilon}{\mathord}{letters}{"07}
\DeclareMathSymbol{\Phi}{\mathord}{letters}{"08}
\DeclareMathSymbol{\Psi}{\mathord}{letters}{"09}
\DeclareMathSymbol{\Omega}{\mathord}{letters}{"0A}

\DeclareMathSymbol{\upGamma}{\mathalpha}{operators}{"00}
\DeclareMathSymbol{\upDelta}{\mathalpha}{operators}{"01}
\DeclareMathSymbol{\upTheta}{\mathalpha}{operators}{"02}
\DeclareMathSymbol{\upLambda}{\mathalpha}{operators}{"03}
\DeclareMathSymbol{\upXi}{\mathalpha}{operators}{"04}
\DeclareMathSymbol{\upPi}{\mathalpha}{operators}{"05}
\DeclareMathSymbol{\upSigma}{\mathalpha}{operators}{"06}
\DeclareMathSymbol{\upUpsilon}{\mathalpha}{operators}{"07}
\DeclareMathSymbol{\upPhi}{\mathalpha}{operators}{"08}
\DeclareMathSymbol{\upPsi}{\mathalpha}{operators}{"09}
\DeclareMathSymbol{\upOmega}{\mathalpha}{operators}{"0A}

\newcommand{\add}[2]{\mathsf{add}_{#1}\mspace{1mu}#2}
\newcommand{\MCM}{\mathsf{MCM}}
\newcommand{\proj}{\mathsf{proj}}
\newcommand{\lproj}[1][A]{#1\text{-}\mathsf{proj}}
\newcommand{\rproj}[1][A]{\mathsf{proj}\text{-}#1}
\newcommand{\Coker}[1]{\operatorname{Coker}#1}
\newcommand{\Ker}[1]{\operatorname{Ker}#1}
\newcommand{\Hom}[3][R]{\operatorname{Hom}_{#1}(#2,#3)}
\newcommand{\RHom}[3][R]{\mathbf{R}\mspace{-1.5mu}\operatorname{Hom}_{#1}(#2,#3)}
\newcommand{\Ext}[4][R]{\operatorname{Ext}_{#1}^{#2}(#3,#4)}
\renewcommand{\mod}{\mathsf{mod}}
\newcommand{\depth}[2][R]{\operatorname{depth}_{#1}#2}

\newcommand{\Ab}{\mathsf{Ab}}
\newcommand{\lMod}[1][A]{#1\text{-}\mathsf{Mod}}
\newcommand{\rMod}[1][A]{\mathsf{Mod}\text{-}#1}
\newcommand{\lmod}[1][A]{#1\text{-}\mathsf{mod}}
\newcommand{\rmod}[1][A]{\mathsf{mod}\text{-}#1}
\newcommand{\dual}{\upOmega}
\newcommand{\gldim}[1]{\operatorname{gldim}\,#1}
\newcommand{\lgldim}[1]{\operatorname{l.gldim}\,#1}
\newcommand{\rgldim}[1]{\operatorname{r.gldim}\,#1}

\begin{document}

\title{The category of maximal Cohen--Macaulay modules as a ring with several objects}

\author{Henrik Holm}

\address{University of Copenhagen, 2100 Copenhagen {\O}, Denmark}
 
\email{holm@math.ku.dk}

\urladdr{http://www.math.ku.dk/\~{}holm/}





\begin{abstract}
  Over a commutative local Cohen--Macaulay ring, we view and study the category of maximal Cohen--Macaulay modules as a ring with several objects. We compute the global dimension of this category and thereby extend a result of Leuschke to the case where the ring has arbitrary (as opposed to finite) CM-type.
\end{abstract}

\maketitle


\section{Introduction}
\label{sec:Introduction}

Let $R$ be a commutative local Cohen--Macaulay ring with Krull dimension $d$. Suppose that $R$ has \emph{finite CM-type}; this means that, up to isomorphism, $R$ admits only finitely many indecomposable maximal Cohen--Macaulay modules $X_1,\ldots,X_n$. In this case, the category $\MCM$ of maximal Cohen--Macaulay $R$-modules has a \emph{representation generator}, i.e.~a module $X \in \MCM$ that contains as direct summands all indecomposable maximal Cohen--Macaulay $R$-modules (for example, $X=X_1 \oplus \cdots \oplus X_n$ would be such a module). A result by Leuschke \cite[Thm.~6]{GJL07} shows that the endomorphism ring $E=\mathrm{End}_R(X)$ has finite global dimension $\leqslant \max\{2,d\}$, and that equality holds if $d \geqslant 2$. This short paper is motivated by Leuschke's result.

If $R$ does not have finite CM-type, then $\MCM$ has no representation generator and there is \emph{a priori} no endomorphism ring $E$ to consider. However, regardless of CM-type, one can always view the entire category $\MCM$ as a ``ring with several objects'' and then study its (finitely presented) left/right ``modules'', i.e.~covariant/contravariant additive functors from $\MCM$ to abelian groups.  The category $\lmod[\MCM]$ of finitely presented left modules over the ``several object ring'' $\MCM$ is the natural object to study in the general case. Indeed, if $R$ has finite CM-type, then this category is equivalent to the category $\lmod[E]$ of finitely generated left $E$-modules, where $E$ is the endomorphism ring introduced above.

It turns out that $\lmod[\MCM]$ and $\rmod[\MCM]$, i.e.~the categories of finitely presented left and right modules over $\MCM$, are abelian with enough projectives. Thus one can naturally speak of the global dimensions of these categories; they are called the left and right global dimensions of $\MCM$, and they are denoted $\lgldim{(\MCM)}$ and $\rgldim{(\MCM)}$. We show that there is an equality $\lgldim{(\MCM)} = \rgldim{(\MCM)}$; this number is simply called the \emph{global dimension of $\MCM$}, and it is denoted by $\gldim{(\MCM)}$. Our first main result, \thmref{gldim-MCM}, shows that there are inequalities,
\begin{displaymath}
  \tag{\text{$*$}}
d \leqslant \gldim{(\MCM)} \leqslant \max\{2,d\}\;,
\end{displaymath}
and thus it extends Leuschke's theorem to the case of arbitrary CM-type. We prove the left inequality in ($*$) by showing that $\MCM$ always admits a finitely presented module with projective dimension $d$. Actually, we show that if $M$ is any Cohen--Macaulay $R$-module of dimension $t$, then
$\Ext{d-t}{M}{-}$ is a finitely presented left $\MCM$-module and $\Hom{-}{M}$ is a finitely presented right $\MCM$-module both with projective dimension equal to $d-t$. Our second main result, \thmref{01}, shows that if $d=0,1$, then the left inequality in ($*$) is an equality if and only if $R$ is regular, that is, there are equivalences:
\begin{align*}
  \gldim{(\MCM)} = 0 &\ \ \iff \ \ \text{$R$ is a field.} \\
  \gldim{(\MCM)} = 1 &\ \ \iff \ \ \text{$R$ is a discrete valuation ring.}
\end{align*}

\section{Preliminaries}

\begin{stp}
  \label{setup}
  Throughout, $(R,\mathfrak{m},k)$ is a commutative noetherian local Cohen--Macaulay ring with Krull dimension $d$. It is assumed that $R$ has a dualizing (or canonical) module $\dual$. 

The category of finitely generated projective $R$-modules is denoted $\proj$; the category of maximal Cohen--Macaulay $R$-modules (defined below) is denoted $\MCM$; and the category of all finitely generated $R$-modules is denoted $\mod$.
\end{stp}

The \emph{depth} of a finitely generated $R$-module $M \neq 0$, denoted $\depth{M}$, is the supremum of the lengths of all $M$-regular sequences $x_1,\ldots,x_n \in \mathfrak{m}$. This numerical invariant can be computed homologically as follows:
\begin{displaymath}
  \depth{M} \,=\, \inf\{\,i \in \mathbb{Z} \,|\, \Ext{i}{k}{M} \neq 0\,\}\;.
\end{displaymath}
By definition, $\depth{0} = \inf\,\emptyset = +\infty$. For a finitely generated $R$-module $M \neq 0$ one always has $\depth{M} \leqslant d$, and $M$ is called \emph{maximal Cohen--Macaulay} if equality holds. The zero module is also considered to be maximal Cohen--Macaulay; thus an arbitrary finitely generated $R$-module $M$ is maximal Cohen--Macaulay if and only if $\depth{M} \geqslant d$.

\begin{ipg}
  \label{MCM-duality}
It is well-known that the dualizing module $\dual$ gives rise to a duality on the category of maximal Cohen--Macaulay modules; more precisely, there is an equivalence of categories:
\begin{displaymath}
  \xymatrix@C=5pc{
    \MCM \ar@<0.6ex>[r]^-{\Hom{-}{\dual}} & \MCM^\mathrm{op}\;. \ar@<0.6ex>[l]^-{\Hom{-}{\dual}} 
  }
\end{displaymath}
We use the shorthand notation $(-)^\dagger$ for the functor $\Hom{-}{\dual}$. For any finitely generated $R$-module $M$ there is a canonical homomorphism $\delta_M \colon M \to M^{\dagger\dagger}$, which is natural in $M$, and because of the equivalence above, $\delta_M$ is an isomorphism if $M$ belongs to $\MCM$.
\end{ipg}

We will need the following results about depth; they are folklore and easily proved\footnote{\ One way to prove \lemref[Lemmas~]{depth-1} and \lemref[]{depth-2} are by induction on $n$ and $m$, using the behaviour of depth on short exact sequences recorded in Bruns and Herzog \cite[Prop.~1.2.9]{BruHer}.}.

\begin{lem}
  \label{lem:depth-1}
Let $n \geqslant 0$ be an integer and let $0 \to X_n \to \cdots \to X_0 \to M \to 0$ be an exact sequence of finitely generated $R$-modules. If $X_0,\ldots,X_n$ are maximal Cohen--Macaulay, then one has $\depth{M} \geqslant d-n$. \qed
\end{lem}


\begin{lem}
  \label{lem:depth-2}
Let $m \geqslant 0$ be an integer and let $\,0 \to K_{m} \to X_{m-1} \to \cdots \to X_0 \to M \to 0$ be an exact sequence of finitely generated $R$-modules. If $X_0,\ldots,X_{m-1}$ are maximal Cohen--Macaulay, then one has $\depth{K_m} \geqslant \min\{d,\depth{M}+m\}$. In particular, if $m \geqslant d$ then the $R$-module $K_m$ is maximal Cohen--Macaulay. \qed
\end{lem}


We will also need a few notions from relative homological algebra.

\begin{dfn}
  \label{dfn:precover}
  Let $\mathcal{A}$ be a full subcategory of a category $\mathcal{M}$. Following Enochs and Jenda \cite[def.~5.1.1]{rha} we say that $\mathcal{A}$ is \emph{precovering} (or \emph{contravariantly finite}) in $\mathcal{M}$ if every $M \in \mathcal{M}$ has an \emph{$\mathcal{A}$-precover} (or a \emph{right $\mathcal{A}$-approximation}); that is, a morphism $\pi \colon A \to M$ with $A \in \mathcal{A}$ such that every other morphism $\pi' \colon A' \to M$ with $A' \in \mathcal{A}$ factors through $\pi$, as illustrated by the following diagram:
\begin{displaymath}
  \xymatrix@R=0.8pc{
    A' \ar@{-->}[dd] \ar[dr]^-{\pi'} & {} \\
    {} & M \\
    A \ar[ur]_-{\pi} & {}
   }
\end{displaymath}
The notion of \emph{$\mathcal{A}$-preenvelopes} (or \emph{left $\mathcal{A}$-approximations}) is categorically dual to the notion defined above. The subcategory $\mathcal{A}$ is said to be \emph{preenveloping} (or \emph{covariantly finite}) in $\mathcal{M}$ if every $M \in \mathcal{M}$ has an $\mathcal{A}$-preenvelope.
\end{dfn}

The following result is a consequence of Auslander and Buchweitz's maximal Cohen--Macaulay approximations. 

\begin{thm}
  \label{thm:MCM-precover}
  Every finitely generated $R$-module has an $\MCM$-precover. 
\end{thm}

\begin{proof}
  By \cite[Thm.~A]{MAsROB89} any finitely generated $R$-module $M$ has a maximal Cohen--Macaulay approximation, that is, a short exact sequence,
\begin{displaymath}
  0 \longrightarrow I \longrightarrow X \stackrel{\pi}{\longrightarrow} M \longrightarrow 0\;,
\end{displaymath}
where $X$ is maximal Cohen--Macaulay and $I$ has finite injective dimension. A classic result of Ischebeck \cite{FIs69} (see also \cite[Exerc. 3.1.24]{BruHer}) shows that $\Ext{1}{X'}{I}=0$ for every $X'$ in $\MCM$, and hence $\Hom{X'}{\pi} \colon \Hom{X'}{X} \to \Hom{X'}{M}$ is surjective.
\end{proof}

\section{Rings with several objects}
\label{sec:Rings-with-several-objects}

The classic references for the theory of rings with several objects are Freyd \cite{abcat, Freyd65} and Mitchell \cite{Mitchell}. Below we recapitulate a few definitions and results that we need.

\vspace*{1.5ex}

A ring $A$ can be viewed as a preadditive category $\bar{A}$ with a single object $\ast$ whose endo hom-set $\Hom[\bar{A}]{\ast}{\ast}$ is $A$, and where composition  is given by ring multiplication. The category $(\bar{A},\Ab)$ of additive covariant functors from $\bar{A}$ to the category $\Ab$ of abelian groups is naturally equivalent to the category $\lMod$ of left $A$-modules. Indeed, an additive functor $F \colon \bar{A} \to \Ab$ yields a left $A$-module whose underlying abelian group is $M=F(\ast)$ and where left $A$-multiplication is given by $am = F(a)(m)$ for $a \in A=\Hom[\bar{A}]{\ast}{\ast}$ and $m \in M=F(\ast)$. Note that the preadditive category associated to the opposite ring $A^\mathrm{o}$ of $A$ is the opposite (or dual) category of $\bar{A}$; in symbols: $\overline{A^\mathrm{o}} = \bar{A}^{\,\mathrm{op}}$. It follows that the category $(\bar{A}^{\,\mathrm{op}},\Ab)$ of additive covariant functors  $\bar{A}^{\,\mathrm{op}} \to \Ab$ (which correspond to additive contravariant functors $\bar{A} \to \Ab$) is naturally equivalent to the category $\rMod$ of right $A$-modules.

These considerations justify the well-known viewpoint that any skeletally small preadditive category $\mathcal{A}$ may be thought of as a \emph{ring with several objects}. A \emph{left $\mathcal{A}$-module} is an additive covariant functor $\mathcal{A}\to\Ab$, and the category of all such is denoted by $\lMod[\mathcal{A}]$. Similarly, a \emph{right $\mathcal{A}$-module} is an additive covariant functor $\mathcal{A}^\mathrm{op}\to\Ab$ (which corresponds to an additive contravariant functor $\mathcal{A} \to \Ab$), and the category of all such is denoted $\rMod[\mathcal{A}]$.

From this point, we assume for simplicity that $\mathcal{A}$ is a skeletally small \emph{additive} category which is \emph{closed under direct summands} (i.e.~every idempotent splits). The category $\lMod[\mathcal{A}]$ is a Grothendieck category, see~\cite[prop.~5.21]{abcat}, with enough projectives. In fact, it follows from Yoneda's lemma that the representable functors $\mathcal{A}(A,-)$, where $A$ is in $\mathcal{A}$, constitute a generating set of projective objects in $\lMod[\mathcal{A}]$. A left $\mathcal{A}$-module $F$ is called \emph{finitely generated}, respectively, \emph{finitely presented} (or \emph{coherent}), if there exists an exact sequence 
$\mathcal{A}(A,-) \to F \to 0$, respectively, $\mathcal{A}(B,-) \to \mathcal{A}(A,-) \to F \to 0$, for some $A,B \in \mathcal{A}$\footnote{ If the category $\mathcal{A}$ is only assumed to be preadditive, then one would have to modify the definitions of finitely generated/presented accordingly. For example, in this case, a left $\mathcal{A}$-module $F$ is called finitely generated if there is an exact sequence of the form $\bigoplus_{i=1}^n\mathcal{A}(A_i,-) \to F \to 0$ for some $A_1,\ldots,A_n \in \mathcal{A}$.}. The category of finitely presented left $\mathcal{A}$-modules is denoted by $\lmod[\mathcal{A}]$. The Yoneda functor,
\begin{displaymath}
\mathcal{A}^\mathrm{op} \longrightarrow \lMod[\mathcal{A}]
\qquad \text{given by} \qquad A \longmapsto \mathcal{A}(A,-)\;,
\end{displaymath} 
is fully faithful, see \cite[thm.~5.36]{abcat}. Moreover, this functor identifies the objects in $\mathcal{A}$ with the finitely generated projective left $\mathcal{A}$-modules, that is, a finitely generated left $\mathcal{A}$-module is projective if and only if it is isomorphic to $\mathcal{A}(A,-)$ for some $A \in \mathcal{A}$; cf.~\cite[exerc.~5-G]{abcat}.

Here is a well-known, but important, example:

\begin{exa}
  \label{exa:proj-mod}
  Let $A$ be any ring and let $\mathcal{A}=\lproj$ be the category of all finitely generated projective left $A$-modules. In this case, the category $\lmod[\mathcal{A}] = \lmod[(\lproj)]$ is equivalent to the category $\rmod[A]$ of finitely presented right $A$-modules. Let us explain why: 

Let $F$ be a left $(\lproj)$-module, that is, an additive covariant functor $F \colon \lproj \to \Ab$. For $a \in A$ the homothety map $\chi_a \colon A \to A$ given by $b \mapsto ba$ is left $A$-linear and so it induces an endomorphism $F(\chi_a)$ of the abelian group $F(A)$. Thus $F(A)$ has a natural structure of a right $A$-module given by $xa = F(\chi_a)(x)$ for $a \in A$ and $x \in F(A)$. This right $A$-module is denoted $\mathsf{e}(F)$, and we get a functor $\mathsf{e}$, called \emph{evaluation}, displayed in the diagram below. The other functor $\mathsf{f}$ in the diagram, called \emph{functorfication}, is given by $\mathsf{f}(M) = M\otimes_A-$ (restricted to $\lproj$) for a right $A$-module $M$.
\begin{equation*}
  \xymatrix@C=4pc{
    \lMod[(\lproj)] \ar@<0.6ex>[r]^-{\mathsf{e}} & \rMod[A] \ar@<0.6ex>[l]^-{\mathsf{f}} 
  }
\end{equation*}

The functors $\mathsf{e}$ and $\mathsf{f}$ yield an equivalence of categories: For every right $A$-module $M$ there is obviously an isomorphism 
$(\mathsf{e} \circ \mathsf{f})(M) = M \otimes_A A \cong M$. We must also show that every left $(\lproj)$-module $F$ is isomorphic to $(\mathsf{f} \circ \mathsf{e})(F) = F(A) \otimes_A -$. For every $P \in \lproj$ and $y \in P$ the left $A$-linear map $\mu_P^y \colon A \to P$ given by $a \mapsto ay$ induces a group homomorphism $F(\mu_P^y) \colon F(A) \to F(P)$, and thus one has a group homomorphism $\tau_P \colon F(A) \otimes_A P \to F(P)$ given by $x \otimes y \mapsto F(\mu_P^y)(x)$. It is straightforward to verify that $\tau$ is a natural transformation. To prove that $\tau_P$ is an isomorphism for every $P \in \lproj$ it suffices, since the functors $F(A) \otimes_A -$ and $F$ are both additive, to check that $\tau_A \colon F(A) \otimes_A A \to F(A)$ is an isomorphism. However, this is evident.

It is not hard to verify that the functors $\mathsf{e}$ and $\mathsf{f}$ restrict to an equivalence between finitely presented objects, as claimed.
\end{exa}

\enlargethispage{5.5ex}

\begin{obs}
  \label{obs:six-equivalences}
  \exaref{proj-mod} shows that for any ring $A$, the category $\lmod[(\lproj)]$ is equivalent to $\rmod[A]$. Since there is an equivalence of categories,
\begin{equation*}
  \xymatrix@C=6pc{
    \lproj \ar@<0.6ex>[r]^-{\Hom[A]{-}{A}} & (\rproj)^\mathrm{op}\;, \ar@<0.6ex>[l]^-{\Hom[A^\mathrm{o}]{-}{A}} 
  }
\end{equation*}
it follows\footnote{\ Cf.~the proof of \prpref{lgldim-rgldim}.} that $\lmod[(\lproj)]$ is further equivalent to $\lmod[((\rproj)^\mathrm{op})]$, which is the same as $\rmod[(\rproj)]$. In conclusion, there are equivalences of categories:
\begin{displaymath}
  \lmod[(\lproj)] \,\simeq\, \rmod[A] \,\simeq\, \rmod[(\rproj)]\;.
\end{displaymath}
Of course, by applying this to the opposite ring $A^\mathrm{o}$ one obtains equivalences:
\begin{displaymath}
  \lmod[(\rproj)] \,\simeq\, \lmod[A] \,\simeq\, \rmod[(\lproj)]\;.
\end{displaymath}
\end{obs}

\vspace*{1.5ex}

In generel, the category $\lmod[\mathcal{A}]$ of finitely presented left $\mathcal{A}$-modules is an additive category with cokernels, but it is not necessarily an abelian subcategory of $\lMod[\mathcal{A}]$. A classic result of Freyd describes the categories $\mathcal{A}$ for which $\lmod[\mathcal{A}]$ is abelian. This result is stated in \thmref{mod-abelian} below, but first we explain some terminology.

A \emph{pseudo-kernel} (also called a \emph{weak kernel}) of a morphism $\beta \colon B \to C$ in $\mathcal{A}$ is a morphism $\alpha \colon A \to B$ such that the sequence
\begin{displaymath}
  \xymatrix@C=2.5pc{
    \mathcal{A}(-,A) \ar[r]^-{\mathcal{A}(-,\alpha)} &
    \mathcal{A}(-,B) \ar[r]^-{\mathcal{A}(-,\mspace{1.5mu}\beta)} &
    \mathcal{A}(-,C)
  }
\end{displaymath}
is exact in $\rMod[\mathcal{A}]$. Equivalently, one has $\beta\alpha=0$ and for every morphism $\alpha' \colon A' \to B$ with $\beta\alpha'=0$ there is a (not necessarily unique!) morphism $\theta \colon A' \to A$ with $\alpha \theta=\alpha'$.
\begin{displaymath}
  \xymatrix@C=2.5pc{
    A \ar[dr]_-{\alpha} \ar@/^2ex/[drr]^-{0} & {} & {} \\
    {} & B \ar[r]^-{\beta} & C \\
    A' \ar@{-->}[uu]^-{\theta} \ar[ur]^-{\alpha'} \ar@/_2ex/[urr]_-{0} & {} & {}     
  }
\end{displaymath}
We say that $\mathcal{A}$ \emph{has pseudo-kernels} is every morphism in $\mathcal{A}$ has a pseudo-kernel.

\emph{Pseudo-cokernels} (also called \emph{weak cokernels}) are defined dually.

\begin{obs}
  \label{obs:Existence-of-pseudo-kernels}
  Suppose that $\mathcal{A}$ is a full subcategory of an abelian category $\mathcal{M}$. 

If $\mathcal{A}$ is precovering in $\mathcal{M}$, see \dfnref{precover}, then $\mathcal{A}$ has pseudo-kernels. Indeed, given a morphism $\beta \colon B \to C$ in $\mathcal{A}$ it has a  kernel $\iota \colon M \to B$ in the abelian category $\mathcal{M}$; and it is easily verified that if $\pi \colon A \to M$ is any $\mathcal{A}$-precover of $M$, then $\alpha=\iota\pi \colon A \to B$ is a pseudo-kernel in $\mathcal{A}$ of $\beta$.

A similar argument shows that if $\mathcal{A}$ is preenveloping in $\mathcal{M}$, then $\mathcal{A}$ has pseudo-cokernels.
\end{obs}

\begin{thm}
  \label{thm:mod-abelian}
  The category $\rmod[\mathcal{A}]$ (respectively, $\lmod[\mathcal{A}]$) of finitely presented right (respectively, left) $\mathcal{A}$-modules is an abelian subcategory of 
$\rMod[\mathcal{A}]$ (respectively, $\lMod[\mathcal{A}]$) if and only if $\mathcal{A}$ has pseudo-kernels (respectively, has pseudo-cokernels).
\end{thm}

\begin{proof}
  See Freyd \cite[thm.~1.4]{Freyd65} or Auslander and Reiten \cite[prop.~1.3]{AR86}.
\end{proof}

\begin{exa}
  Let $A$ be a left and right noetherian ring. As $A$ is left noetherian, the category $\mathcal{M}=\lmod[A]$ of finitely presented left $A$-modules is abelian, and evidently $\mathcal{A}=\lproj$ is precovering herein. As $A$ is right noetherian, $\lproj$ is also preenveloping in $\lmod[A]$; cf.~\cite[Exa.~8.3.10]{rha}. It follows from \obsref{Existence-of-pseudo-kernels} that $\lproj$ has both pseudo-kernels and pseudo-cokernels, and therefore the categories $\rmod[(\lproj)]$ and $\lmod[(\lproj)]$ are abelian by \thmref{mod-abelian}. Of course, this also follows directly from \obsref{six-equivalences} which shows that $\rmod[(\lproj)]$ and $\lmod[(\lproj)]$ are equivalent to $\lmod[A]$ and $\rmod[A]$, respectively.
\end{exa}

Note that if $\lmod[\mathcal{A}]$ is abelian, i.e.~if $\mathcal{A}$ has pseudo-cokernels, then every finitely presented left $\mathcal{A}$-module $F$ admits a projective resolution in $\lmod[\mathcal{A}]$, that is, an exact sequence
\begin{displaymath}
  \cdots \longrightarrow \mathcal{A}(A_1,-)  \longrightarrow \mathcal{A}(A_0,-) \longrightarrow F \longrightarrow 0
\end{displaymath}
where $A_0,A_1,\ldots$ belong to $\mathcal{A}$. Thus one can naturally speak of the \emph{projective dimension} of $F$ (i.e.~the length, possibly infinite, of the shortest projective resolution of $F$ in $\lmod[\mathcal{A}]$) and of the \emph{global dimension} of the category $\lmod[\mathcal{A}]$ (i.e.~the supremum of projective dimensions of all objects in $\lmod[\mathcal{A}]$).

\begin{dfn}
  \label{dfn:gldim}
  In the case where the category $\lmod[\mathcal{A}]$ (respectively, $\rmod[\mathcal{A}]$) is abelian, then its global dimension is called the \emph{left} (respectively, \emph{right}) \emph{global dimension of $\mathcal{A}$}, and it is denoted $\lgldim{\mathcal{A}}$ (respectively, $\rgldim{\mathcal{A}}$).
\end{dfn}

Note that $\lgldim{(\mathcal{A}^\mathrm{op})}$ is the same as $\rgldim{\mathcal{A}}$ (when these numbers make sense).

\begin{exa}
  \label{exa:gldim-proj}
  Let $A$ be a left and right noetherian ring whose global dimension\footnote{\ Recall that for a ring which is both left and right noetherian, the left and right global dimensions are equal; indeed, they both coincide with the weak global dimension.} we denote $\gldim{A}$. Recall that $\gldim{A}$ can be computed as the supremum of projective dimensions of all \emph{finitely generated} (left or right) $A$-modules. It follows from \obsref{six-equivalences} that
\begin{displaymath}
  \lgldim{(\lproj)} \,=\, \gldim{A} \,=\, \rgldim{(\lproj)}\;.
\end{displaymath}
\end{exa}

\section{The global dimension of the category $\MCM$}

We are now in a position to prove the results announced in the Introduction.

\begin{exa}
  \label{exa:finite-CM-type}
  Suppose that $R$ has finite CM-type and let $X$ be any representation generator of the category $\MCM$, cf.~\secref{Introduction}. This means that $\MCM = \add{R}{X}$ where $\add{R}{X}$ denotes the category of direct summands of finite direct sums of copies of $X$.
Write $E = \mathrm{End}_R(X)$ for the endomorphism ring of $X$; this $R$-algebra is often referred to as the \emph{Auslander algebra}. Note that $X$ has a canonical structure as a left-$R$--left-$E$--bimodule ${}_{R,E}X$. It is easily verified that there is an equivalence, known as Auslander's \emph{projectivization}, given by:
\begin{equation*}
  \xymatrix@C=6pc{
    \MCM = \add{R}{X} \ar@<0.6ex>[r]^-{\Hom{X}{-}} & \rproj[E]\;. \ar@<0.6ex>[l]^-{- \otimes_E X} 
  }
\end{equation*}
It now follows from \obsref{six-equivalences} that there are equivalences of categories:
\begin{displaymath}
  \lmod[\MCM] \,\simeq\, \lmod[{(\rproj[E])}] \,\simeq\, \lmod[E]\;.
\end{displaymath}
Similarly, there is an equivalence of categories: $\rmod[\MCM] \simeq \rmod[E]$.
\end{exa}

\begin{prp}
  \label{prp:MCM-pseudo}
  The category $\MCM$ has pseudo-kernels and pseudo-cokernels. 
\end{prp}

\begin{proof}
  As $\MCM$ is precovering in the abelian category $\mod$, see \thmref{MCM-precover}, we get from \obsref{Existence-of-pseudo-kernels} that $\MCM$ has pseudo-kernels. To prove that $\MCM$ has pseudo-cokernels, let $\alpha \colon X \to Y$ be any homomorphism between maximal Cohen--Macaulay $R$-modules. With the notation from \ref{MCM-duality} we let $\iota \colon Z \to Y^\dagger$ be a pseudo-kernel in $\MCM$ of $\alpha^\dagger \colon Y^\dagger \to X^\dagger$. We claim that $\iota^\dagger\delta_Y \colon Y \to Z^\dagger$ is a pseudo-cokernel of $\alpha$, i.e.~that the sequence
\begin{equation}
  \label{eq:pseudo-cokernel-1}
  \xymatrix@C=4pc{
    \Hom{Z^\dagger}{U} \ar[r]^-{\Hom{\iota^\dagger\delta_Y}{U}} &
    \Hom{Y}{U} \ar[r]^-{\Hom{\alpha}{U}} &
    \Hom{X}{U}
  }
\end{equation}
is exact for every $U \in \MCM$. From the commutative diagram
\begin{displaymath}
  \xymatrix{
    X \ar[d]^-{\delta_X}_-{\cong} \ar[r]^-{\alpha} & 
    Y \ar[d]^-{\delta_Y}_-{\cong} \ar[r]^-{\iota^\dagger\delta_Y} & 
    Z^\dagger \ar@{=}[d] \\
    X^{\dagger\dagger} \ar[r]^-{\alpha^{\dagger\dagger}} & Y^{\dagger\dagger} \ar[r]^-{\iota^\dagger} & Z^\dagger \\
  }
\end{displaymath}
it follows that the sequence \eqref{pseudo-cokernel-1} is isomorphic to 
\begin{equation}
  \label{eq:pseudo-cokernel-2}
  \xymatrix@C=4pc{
    \Hom{Z^\dagger}{U} \ar[r]^-{\Hom{\iota^\dagger}{U}} &
    \Hom{Y^{\dagger\dagger}}{U} \ar[r]^-{\Hom{\alpha^{\dagger\dagger}}{U}} &
    \Hom{X^{\dagger\dagger}}{U}
  }.
\end{equation}
Recall from \ref{MCM-duality} that there is an isomorphism $U \cong U^{\dagger\dagger}$. From this fact and from the ``swap'' isomorphism \cite[(A.2.9)]{lnm}, it follows that the sequence \eqref{pseudo-cokernel-2} is isomorphic to
\begin{equation}
  \label{eq:pseudo-cokernel-3}
  \xymatrix@C=4.5pc{
    \Hom{U^\dagger}{Z^{\dagger\dagger}} \ar[r]^-{\Hom{U^\dagger}{\iota^{\dagger\dagger}}} &
    \Hom{U^\dagger}{Y^{\dagger\dagger\dagger}} \ar[r]^-{\Hom{U^\dagger}{\alpha^{\dagger\dagger\dagger}}} &
    \Hom{U^\dagger}{X^{\dagger\dagger\dagger}}
  }.
\end{equation}
Finally, the commutative diagram
\begin{displaymath}
  \xymatrix{
    Z \ar[d]^-{\delta_Z}_-{\cong} \ar[r]^-{\iota} & 
    Y^\dagger \ar[d]^-{\delta_{Y^\dagger}}_-{\cong} \ar[r]^-{\alpha^\dagger} & 
    X^\dagger \ar[d]^-{\delta_{X^\dagger}}_-{\cong} \\
    Z^{\dagger\dagger} \ar[r]^-{\iota^{\dagger\dagger}} & Y^{\dagger\dagger\dagger} \ar[r]^-{\alpha^{\dagger\dagger\dagger}} & X^{\dagger\dagger\dagger} \\
  }
\end{displaymath}
shows that the sequence \eqref{pseudo-cokernel-3} is isomorphic to
\begin{equation*}
  \xymatrix@C=3.5pc{
    \Hom{U^\dagger}{Z} \ar[r]^-{\Hom{U^\dagger}{\iota}} &
    \Hom{U^\dagger}{Y^\dagger} \ar[r]^-{\Hom{U^\dagger}{\alpha^\dagger}} &
    \Hom{U^\dagger}{X^\dagger}
  },
\end{equation*}
which is exact since $\iota \colon Z \to Y^\dagger$ is a pseudo-kernel of $\alpha^\dagger \colon Y^\dagger \to X^\dagger$.
\end{proof}

We shall find the following notation useful.

\begin{dfn}
For an $R$-module $M$, we use the notation $(M,-)$ for the left $\MCM$-module $\Hom{M}{-}|_\MCM$, and $(-,M)$ for the right $\MCM$-module $\Hom{-}{M}|_\MCM$.
\end{dfn}

\thmref{mod-abelian} and \prpref{MCM-pseudo} shows that $\lmod[\MCM]$ and $\rmod[\MCM]$ are abelian, and hence the left and right global dimensions of the category $\MCM$ are both well-defined; see \dfnref{gldim}. In fact, they are equal:

\begin{prp}
  \label{prp:lgldim-rgldim}
  The left and right global dimensions of $\MCM$ coincide, that is, 
\begin{displaymath}
  \lgldim{(\MCM)} = \rgldim{(\MCM)}\;.
\end{displaymath}
  This number is called the \emph{\emph{global dimension}} of $\MCM$, and it is denoted $\gldim{(\MCM)}$. \qed
\end{prp}

\begin{proof}
  The equivalence in \ref{MCM-duality} induces an equivalence between the abelian categories of (all) left and right $\MCM$-modules given by:
\begin{equation}
  \label{eq:lMod-rMod}
  \xymatrix@C=6pc{
    \lMod[\MCM] \ar@<0.6ex>[r]^-{F \ \mapsto \ F \,\circ\, (-)^\dagger} & \rMod[\MCM]\;. \ar@<0.6ex>[l]^-{G \,\circ\, (-)^\dagger \ \mapsfrom \ G} 
  }
\end{equation}
These functors preserve finitely generated projective modules. Indeed, if $P = (X,-)$ with $X \in \MCM$ is a finitely generated projective left $\MCM$-module, then the right $\MCM$-module $P \circ (-)^\dagger = (X,(-)^\dagger)$ is isomorphic to $(-,X^\dagger)$, which is finitely generated projective. Similarly, if $Q = (-,Y)$ with $Y \in \MCM$ is a finitely generated projective right $\MCM$-module, then $Q \circ (-)^\dagger = ((-)^\dagger,Y)$ is isomorphic to $(Y^\dagger,-)$, which is finitely generated projective. 

Since the functors in \eqref{lMod-rMod} are exact and preserve finitely generated projective modules, they restrict to an equivalence between finitely presented objects, that is, $\lmod[\MCM]$ and $\rmod[\MCM]$ are equivalent. It follows that $\lmod[\MCM]$ and $\rmod[\MCM]$ have the same global dimension, i.e.~the left and right global dimensions of $\MCM$ coincide.
\end{proof}

We begin our study of $\gldim{(\MCM)}$ with a couple of easy examples.

\begin{exa}
  \label{exa:gldim-MCM-regular}
  If $R$ is regular, in which case the global dimension of $R$ is equal to $d$, then one has $\MCM = \proj$, and it follows from \exaref{gldim-proj} that $\gldim{(\MCM)} = d$.
\end{exa}

\begin{exa}
  \label{exa:gldim-MCM}
  Assume that $R$ has finite CM-type and denote the Auslander algebra by $E$. It follows from \exaref{finite-CM-type} that $\gldim{(\MCM)} = \gldim{E}$.
\end{exa}

We turn our attention to projective dimensions of representable right $\MCM$-modules.

\begin{prp}
  \label{prp:pd-of-contravariant-Hom}
  Let $M$ be a finitely generated $R$-module. Then $(-,M)$ is a finitely presented right $\MCM$-module with projective dimension equal to $d-\depth{M}$.
\end{prp}

\begin{proof}
  First we argue that $(-,M)$ is finitely presented. By \thmref{MCM-precover} there is an $\MCM$-precover $\pi \colon X \to M$, which by definition yields an epimorphism $(-,\pi) \colon (-,X) \twoheadrightarrow (-,M)$. Hence $(-,M)$ is finitely generated. As the Hom functor is left exact, the kernel of $(-,\pi)$ is the functor $(-,\Ker{\pi})$. Since $\Ker{\pi}$ is a finitely generated $R$-module, the argument above shows that $(-,\Ker{\pi})$ is finitely generated, and therefore $(-,M)$ is finitely presented.

If $M=0$, then $(-,M)$ is the zero functor which has projective dimension $d-\depth{M}=-\infty$. Thus we can assume that $M$ is non-zero such that $m:=d-\depth{M}$ is an integer. By 
successively taking $\MCM$-precovers, whose existence is guaranteed by \thmref{MCM-precover}, we construct an exact sequence of $R$-modules, $0 \to K_m \to X_{m-1} \to \cdots \to X_1 \to X_0 \to M \to 0$, where $X_0,\ldots,X_{m-1}$ are maximal Cohen--Macaulay and $K_m = \Ker{(X_{m-1} \to X_{m-2})}$, such that the sequence
\begin{equation*}
  \label{eq:proj-res-of-functor}
  \xymatrix@C=1pc{
    0 \ar[r] & (-,K_m) \ar[r] & (-,X_{m-1}) \ar[r] & \cdots \ar[r] & 
    (-,X_1) \ar[r] & 
    (-,X_0) \ar[r] & (-,M) \ar[r] & 0
  }
\end{equation*}
in $\rmod[\MCM]$ is exact. \lemref{depth-2} show that $\depth{K_m} \geqslant \min\{d,\depth{M}+m\} = d$, and hence $K_m$ is maximal Cohen--Macaulay. Thus exactness of the sequence displayed above shows that the projective dimension of $(-,M)$ is $\leqslant m$.

To prove that the projective dimension of $(-,M)$ is $\geqslant m$, we must show that if
\begin{displaymath}
  \xymatrix@C=1.5pc{
    0 \ar[r] & 
    (-,Y_{n}) \ar[r]^-{\tau_{n}} & \cdots \ar[r] & 
    (-,Y_1) \ar[r]^-{\tau_1} & 
    (-,Y_0) \ar[r]^-{\tau_0} & (-,M) \ar[r] & 0
  }
\end{displaymath}
is any exact sequence in $\rmod[\MCM]$, where $Y_0,\ldots,Y_n$ are maximal Cohen--Macaulay, then $n \geqslant m$. By Yoneda's lemma, each $\tau_i$ has the form $\tau_i = (-,\beta_i)$ for some  homomorphism $\beta_i \colon Y_i \to Y_{i-1}$ when $1 \leqslant i \leqslant n$ and $\beta_0 \colon Y_0 \to M$. By evaluating the sequence on the maximal Cohen--Macaulay module $R$, it follows that the sequence of $R$-modules,
\begin{displaymath}
  \xymatrix@C=1.5pc{
    0 \ar[r] & Y_{n} \ar[r]^-{\beta_{n}} & \cdots \ar[r] & 
    Y_1 \ar[r]^-{\beta_1} & 
    Y_0 \ar[r]^-{\beta_0} & M \ar[r] & 0
  },
\end{displaymath}
is exact. Thus \lemref{depth-1} yields $\depth{M} \geqslant d-n$, that is, $n \geqslant m$.
\end{proof}

In contrast to what is the case for representable right $\MCM$-modules, representable left $\MCM$-modules are ``often'' zero. For example, if $d>0$ then $\Hom{k}{X}=0$ for every maximal Cohen--Macaulay $R$-module $X$, and hence $(k,-)$ is the zero functor. In particular, the 
projective dimension of a representable left $\MCM$-module is typically not very interesting. \prpref{Ext} below gives concrete examples of finitely presented left $\MCM$-modules that do have interesting projective dimension.

\begin{lem}
  \label{lem:natural-iso}
  For every Cohen--Macaulay $R$-module $M$ of dimension $t$ there is the following natural isomorphism of functors $\MCM \to \Ab$,
  \begin{align*}
    \Hom{(-)^\dagger}{\Ext{d-t}{M}{\dual}} \,\cong\, \Ext{d-t}{M}{-}\;.
  \end{align*}
\end{lem}

\begin{proof}
  Since $M$ is Cohen--Macaulay of dimension $t$ one has $\Ext{i}{M}{\dual} = 0$ for $i \neq d-t$; see \cite[Thm.~3.3.10]{BruHer}. Thus there is an isomorphism in the derived category of $R$,
\begin{displaymath}
  \Ext{d-t}{M}{\dual} \cong \upSigma^{d-t}\RHom{M}{\dual}\;.
\end{displaymath}
In particular, there is an isomorphism $X^\dagger=\Hom{X}{\dual} \cong \RHom{X}{\dual}$ for $X \in \MCM$. This explains the first isomorphism below. The second isomorphism is trivial, the third one is by ``swap'' \cite[(A.4.22)]{lnm}, and the fourth one follows as $\dual$ is a dualizing $R$-module.
\begin{align*}
  \RHom{X^\dagger}{\Ext{d-t}{M}{\dual}} 
  &\cong \RHom{\RHom{X}{\dual}}{\upSigma^{d-t}\RHom{M}{\dual}} \\
  &\cong \upSigma^{d-t}\RHom{\RHom{X}{\dual}}{\RHom{M}{\dual}} \\
  &\cong \upSigma^{d-t}\RHom{M}{\RHom{\RHom{X}{\dual}}{\dual}} \\
  &\cong \upSigma^{d-t}\RHom{M}{X}\;.
\end{align*}
The assertion now follows by taking the zero'th homology group $\mathrm{H}_0$.
\end{proof}

\begin{prp}
  \label{prp:Ext}
  If $M$ is any Cohen--Macaulay $R$-module of dimension $t$, then the~functor $\Ext{d-t}{M}{-}|_\MCM$ is a finitely presented left $\MCM$-module with projective dimension equal to $d-t$.
\end{prp}

\begin{proof}
  As $M$ is Cohen--Macaulay of dimension $t$, so is \smash{$\Ext{d-t}{M}{\dual}$}; see \cite[Thm.~3.3.10]{BruHer}. \prpref{pd-of-contravariant-Hom} shows that \smash{$\Hom{-}{\Ext{d-t}{M}{\dual}}|_\MCM$} is a finitely presented right $\MCM$-module with projective dimension equal to $d-t$. The proof of \prpref[Prop.~]{lgldim-rgldim} now gives that 
\begin{displaymath}
  \Hom{(-)^\dagger}{\Ext{d-t}{M}{\dual}}|_\MCM
\end{displaymath}
is a finitely presented left $\MCM$-module with projective dimension $d-t$, and   \lemref{natural-iso} finishes the proof. 
\end{proof}

\begin{thm}
  \label{thm:gldim-MCM}
  The category $\MCM$ has finite global dimension. In fact, one has 
\begin{displaymath}
d \leqslant \gldim{(\MCM)} \leqslant \max\{2,d\}\;.
\end{displaymath}
In particular, if $d \geqslant 2$ then there is an equality $\gldim{(\MCM)}=d$.
\end{thm}

\begin{proof}
  The residue field $k$ of $R$ is a finitely generated $R$-module with depth $0$. Thus \prpref{pd-of-contravariant-Hom} shows that $(-,k)$ is finitely presented right $\MCM$-module with projective dimension $d$. Consequently, we must have $d \leqslant \gldim{(\MCM)}$.

  To prove the other inequality, set $m = \max\{2,d\}$ and let $G$ be any finitely presented right $\MCM$-module. Take any exact sequence in $\rmod[\MCM]$, 
\begin{equation}
  \label{eq:resolution}
  \xymatrix@C=1.5pc{
    (-,X_{m-1}) \ar[r]^-{\tau_{m-1}} & \cdots \ar[r] & 
    (-,X_1) \ar[r]^-{\tau_1} & 
    (-,X_0) \ar[r]^-{\varepsilon} & G \ar[r] & 0
  },
\end{equation}
where $X_0,X_1,\ldots,X_{m-1}$ are in $\MCM$. Note that since $m \geqslant 2$ there is at least one ``$\tau$'' in this sequence. By Yoneda's lemma, every $\tau_i$ has the form $\tau_i = (-,\alpha_i)$ for some homomorphism $\alpha_i \colon X_i \to X_{i-1}$. By evaluating \eqref{resolution} on $R$, we get an exact sequence of $R$-modules:
\begin{displaymath}
  \xymatrix@C=1.5pc{
    X_{m-1} \ar[r]^-{\alpha_{m-1}} & \cdots \ar[r] & 
    X_1 \ar[r]^-{\alpha_1} & 
    X_0
  }.
\end{displaymath}
Since $m \geqslant d$ it follows from \lemref{depth-2} that the module $X_m = \Ker{\alpha_{m-1}}$ is maximal Cohen--Macaulay. As the Hom functor is left exact, we see that $0 \to (-,X_m) \to (-,X_{m-1})$ is exact. This sequence, together with \eqref{resolution}, shows that $G$ has projective dimension $\leqslant m$.
\end{proof}

In view of \exaref{gldim-MCM} and \thmref{gldim-MCM}, we immediate get the following result due to Leuschke \cite[Thm.~6]{GJL07}.

\begin{cor}
  \label{cor:Leuschke}
  Assume that $R$ has finite CM-type and let $X$ be any representation generator of $\MCM$ with Auslander algebra $E = \mathrm{End}_R(X)$. There are inequalities,
\begin{displaymath}
  d \leqslant \gldim{E} \leqslant \max\{2,d\}\;.
\end{displaymath}
In particular, if $d \geqslant 2$ then there is an equality $\gldim{E}=d$. \qed
\end{cor}

\begin{exa}
  If $d=0$ then $\MCM = \mod$ and hence $\gldim{(\MCM)} = \gldim{(\mod)}$. Since $\mod$ is abelian, it is a well-known result of Auslander \cite{MAs66} that the latter number must be either $0$ or $2$ (surprisingly, it can not be $1$). Thus one of the inequalities in \thmref{gldim-MCM} is actually an equality. If, for example, $R=k[x]/(x^2)$ where $k$ is a field, then $\gldim{(\mod)}=2$.
\end{exa}

\begin{exa}
  If $d=1$ then \thmref{gldim-MCM} shows that $\gldim{(\MCM)} = 1,2$. The $1$-dimen\-sion\-al Cohen--Macaulay ring \mbox{$R=k[\mspace{-2.7mu}[x,y]\mspace{-2.7mu}]/(x^2)$} does not have finite CM-type\footnote{\ See Buchweitz, Greuel, and Schreyer \cite[Prop.~4.1]{BGS-87} for a complete list of the indecomposable maximal Cohen--Macaulay modules over this ring.}, and since it is not regular, it follows from \thmref{01} below that $\gldim{(\MCM)}=2$. 
\end{exa}

Recall that in any abelian category with enough projectives (such as $\rmod[\mathcal{A}]$ in the case where $\mathcal{A}$ has pseudo-kernels) one can well-define and compute Ext in the usual way.

\begin{lem}
  \label{lem:Ext2}
  Assume that $\mathcal{A}$ is precovering in an abelian category $\mathcal{M}$ (in which case the category $\rmod[\mathcal{A}]$ is abelian by \obsref{Existence-of-pseudo-kernels} and \thmref{mod-abelian}). Let 
\begin{displaymath}
0 \longrightarrow A' \stackrel{\alpha'}{\longrightarrow} A \stackrel{\alpha}{\longrightarrow} A'' 
\end{displaymath}
be an exact sequence in $\mathcal{M}$ where $A$, $A'$, $A''$ belong to  $\mathcal{A}$. Consider the finitely presented right $\mathcal{A}$-module $G = \Coker{\mathcal{A}(-,\alpha)}$, that is, $G$ is defined by exactness of the sequence
\begin{displaymath}
  \xymatrix{
    \mathcal{A}(-,A) \ar[r]^-{\mathcal{A}(-,\alpha)} & \mathcal{A}(-,A'') \ar[r] & G \ar[r] & 0
  }.
\end{displaymath}
For any finitely presented right $\mathcal{A}$-module $H$ there is an isomorphism of abelian groups,
\begin{displaymath}
  \Ext[{\rmod[\mathcal{A}]}]{2}{G}{H} \,\cong\, \Coker{H(\alpha')}\;.
\end{displaymath}
\end{lem}

\begin{proof}
  By the definition of $G$ and left exactness of the Hom functor, the chain complex
\begin{equation}
  \label{eq:proj-res}
  \xymatrix{
    0 \ar[r] & \mathcal{A}(-,A') \ar[r]^-{\mathcal{A}(-,\alpha')} & \mathcal{A}(-,A) \ar[r]^-{\mathcal{A}(-,\alpha)} & \mathcal{A}(-,A'') \ar[r] & 0
  },
\end{equation}
is a non-augmented projective resolution in $\rmod[\mathcal{A}]$ of $G$. To compute $\Ext[{\rmod[\mathcal{A}]}]{2}{G}{H}$ we must first apply the functor $(\rmod[\mathcal{A}])(\mspace{1.5mu}?\mspace{1.5mu},H)$ to \eqref{proj-res} and then take the second cohomology group of the resulting cochain complex. By Yoneda's lemma there is a natural isomorphism
\begin{displaymath}
  (\rmod[\mathcal{A}])(\mathcal{A}(-,B),H) \,\cong\, H(B)
\end{displaymath}
for any $B \in \mathcal{A}$; hence application of $(\rmod[\mathcal{A}])(\mspace{1.5mu}?\mspace{1.5mu},H)$ to \eqref{proj-res} yields the cochain complex
\begin{displaymath}
  \xymatrix@C1.7pc{
    0 \ar[r] & H(A'') \ar[r]^-{H(\alpha)} & H(A) \ar[r]^-{H(\alpha')} & H(A') \ar[r] & 0
  }.
\end{displaymath}
The second cohomology group of this cochain complex is $\Coker{H(\alpha')}$.
\end{proof}

Recall that a commutative ring is called a \emph{discrete valuation ring (DVR)} if it is a principal ideal domain with exactly one non-zero maximal ideal. There are of course many other equivalent characterizations of such rings.

\begin{thm}
  \label{thm:01}
  If $\,\gldim{(\MCM)} \leqslant 1$, then $R$ is regular. In particular, one has
\begin{align*}
  \gldim{(\MCM)} = 0 &\ \ \iff \ \ \text{$R$ is a field.} \\
  \gldim{(\MCM)} = 1 &\ \ \iff \ \ \text{$R$ is a discrete valuation ring.}
\end{align*}
\end{thm}

\begin{proof}
Assume that $\gldim{(\MCM)} \leqslant 1$. Let $X$ be any maximal Cohen--Macaulay $R$-module and let $\pi \colon L \twoheadrightarrow X$ be an epimorphism where $L$ is finitely generated and free. Note that $Y = \Ker{\pi}$ is also maximal Cohen--Macaulay by \lemref{depth-2}, so we have an exact sequence,
\begin{displaymath}
  0 \longrightarrow Y \stackrel{\iota}{\longrightarrow} L \stackrel{\pi}{\longrightarrow} X \longrightarrow 0\;,
\end{displaymath}
of maximal Cohen--Macaulay $R$-modules. With $G = \Coker{(-,\pi)}$ and $H = (-,Y)$ we have
\begin{displaymath}
  \Coker{\mspace{1mu}(\iota,Y)} \,\cong\, \Ext[{\rmod[\MCM]}]{2}{G}{H} \,\cong\, 0 \;;
\end{displaymath} 
here the first isomorphism comes from \lemref{Ext2}, and the second isomorphism follows from the assumption that $\gldim{(\MCM)} \leqslant 1$. Hence the homomorphism
\begin{displaymath}
  \xymatrix@C3.5pc{
    \Hom{L}{Y} \ar[r]^-{\Hom{\iota}{Y}} & \Hom{Y}{Y}
  }
\end{displaymath} 
is surjective. Thus $\iota$ has a left inverse and $X$ becomes a direct summand of the free module $L$. Therefore every maximal Cohen--Macaulay $R$-module is projective, so $R$ is regular.

The displayed equivalences now follows in view of \exaref{gldim-MCM-regular} and the fact that a regular local ring has Krull dimension $0$, respectively, $1$, if and only if it is a field, respectively, a discrete valuation ring.
\end{proof}

As as corollary, we get the following addendum to Leuschke's theorem (see   \corref[]{Leuschke}).

\begin{cor}
  Assume that $R$ has finite CM-type and let $X$ be any representation generator of $\,\MCM$ with Auslander algebra $E = \mathrm{End}_R(X)$. If $\gldim{E} \leqslant 1$, then $R$ is regular. \qed
\end{cor}

\def\cprime{$'$}
  \providecommand{\arxiv}[2][AC]{\mbox{\href{http://arxiv.org/abs/#2}{\sf
  arXiv:#2 [math.#1]}}}
  \providecommand{\oldarxiv}[2][AC]{\mbox{\href{http://arxiv.org/abs/math/#2}{\sf
  arXiv:math/#2
  [math.#1]}}}\providecommand{\MR}[1]{\mbox{\href{http://www.ams.org/mathscinet-getitem?mr=#1}{#1}}}
  \renewcommand{\MR}[1]{\mbox{\href{http://www.ams.org/mathscinet-getitem?mr=#1}{#1}}}
\providecommand{\bysame}{\leavevmode\hbox to3em{\hrulefill}\thinspace}
\providecommand{\MR}{\relax\ifhmode\unskip\space\fi MR }
\providecommand{\MRhref}[2]{%
  \href{http://www.ams.org/mathscinet-getitem?mr=#1}{#2}
}
\providecommand{\href}[2]{#2}

\end{document}